\documentclass{article} 
\usepackage{arxiv}
\usepackage{adjustbox}
\usepackage{pdflscape}
\usepackage[english]{babel} 
\usepackage{amssymb}
\usepackage{amsmath}
\usepackage{amsthm}
\usepackage{mathtools}

\usepackage{abstract}

\usepackage{xcolor}
\usepackage{breqn,url}
\usepackage{tikz}
\usepackage{pgfplots}
\usepackage{graphicx}
\pgfplotsset{width=8cm,compat=1.9}
\usepackage{caption}
\usepackage{subcaption}
\usepackage{etoolbox}

\makeatletter
\AtBeginEnvironment{figure}{%
  \def\@floatboxreset{\reset@font\small\@setminipage}%
}
\makeatother

\newtheorem{lemma}{Lemma}

\newtheorem{proposition}{Proposition}

\newtheorem{example}{Example}

\newcommand{\K}{{\mathbf K}}
\newcommand{\R}{{\mathbb R}}
\newcommand{\N}{{\mathbb N}}
\newtheorem{theorem}{Theorem}
\newtheorem{corollary}{Corollary}

\newcommand{\jacksoncoef}[2]{\lambda_{#1}^{#2}}
\newcommand{\jackpoly}{K^{\rm ja}}
\newcommand{\multijackpoly}{K}

\newcommand{\x}{\mathbf{x}}
\newcommand{\y}{\mathbf{y}}
\newcommand{\varu}{\mathbf{u}}

\newcommand{\rev}[1]{{\color{black}{#1}}}
\newcommand{\revv}[1]{{\color{black}{#1}}}

\title{The link between $1$-norm approximation and effective Positivstellensatze for the hypercube}

\author{
	\textbf{Etienne de Klerk}
	\thanks{Tilburg University,	e.deklerk@tilburguniversity.edu
	}	
	\and
\textbf{Juan C. Vera}
	\thanks{Tilburg University,	j.c.veralizcano@tilburguniversity.edu   }
}

\begin{document}
	\maketitle
	
	\begin{abstract}
The Schm\"udgen Positivstellensatz gives a certificate to verify positivity of a strictly positive polynomial $f$ on a compact, basic, semi-algebraic set $\mathbf{K} \subset \R^n$.
A Positivstellensatz of this type is called {\em  effective} if one may bound the degrees of the polynomials appearing in the
certificate in terms of properties of $f$.
If $\mathbf{K} = [-1,1]^n$ and $0 < f_\min := \min_{x \in \mathbf{K}} f(x)$, then the degrees of the polynomials appearing in the certificate may be bounded by $O\left(\sqrt{\frac{f_\max - f_\min}{f_\min}}\right)$,
where $f_\max := \max_{x \in \mathbf{K}} f(x)$,
as was recently shown by Laurent and Slot [\emph{Optimization Letters} 17:515--530, 2023].
The big-O notation suppresses dependence on $n$ and the degree $d$ of $f$.
In this paper we show a similar result, but with a better dependence on $n$ and $d$. In particular, our bounds depend on the $1$-norm of the
coefficients of $f$, that may readily be calculated.

		\keywords{Polynomial kernel method \and semidefinite programming \and Positivstellensatz}
	\end{abstract}
	
	\section{Introduction}

We first recall a \emph{Positivstellensatz} of Schm\"udgen \cite{Sch91}, which
gives representations for positive polynomials on a basic, closed, semi-algebraic set; see also, e.g., \cite{DP},\cite[Theorem 3.16]{sos}, \cite{Mar}.
\begin{theorem}[Schm\"udgen \cite{Sch91}]
\label{thm:Schmuedgen}
Consider the set   $\mathbf{K} = \{{\mathbf{x}} \in \mathbb{R}^n\mid  g_1(\mathbf{x})\geq 0,\dotsc,g_m(\mathbf{x}) \geq 0\}$
with $g_1,\dotsc,g_m \in\mathbb{R}[{\mathbf{x}}]$, where $\mathbb{R}[\mathbf{x}]$ denotes the ring of real polynomials
 with variables $\mathbf{x} = (x_1,\ldots,x_n)$.
	Assume that $\mathbf{K}$ is  compact.
 If $p \in \mathbb{R}[x]$ is positive on $\mathbf{K}$, then $p$ can be written as
 \begin{equation}\label{Schmuedgen_representation}
    p = \sum_{I \subseteq [m]} \sigma_I \prod_{i\in I} g_i,
 \end{equation}
where $[m] := \{1,\ldots,m\}$, and each $\sigma_I$ ($I \subseteq [m]$) belongs to the cone of sums-of-squares of polynomials,  denoted by $\Sigma[\mathbf{x}]$.
\end{theorem}
A polynomial that allows a representation of the type \eqref{Schmuedgen_representation} is said to lie in the pre-ordering associated with the (generating) polynomials $g_1,\dotsc,g_m$, denoted by
$\mathcal{T}(g_1,\ldots,g_m)$.

We denote the real polynomials of degree at most $d$ by $\mathbb{R}[\mathbf{x}]_d$, and define the
truncated pre-ordering of degree $r \in \mathbb{N}$ by
\begin{equation}\label{eq:trubcated preordering def}
  \mathcal{T}(g_1,\ldots,g_m)_r := \left\{ p \in \mathbb{R}[\mathbf{x}] \; : \; p = \sum_{I \subseteq [m]} \sigma_I \prod_{i\in I} g_i, \; \mbox{deg}\left(\sigma_I \prod_{i\in I} g_i\right) \le r, \; \sigma_I \in \Sigma[x] \; \forall I \in [m]\right\}.
\end{equation}
For the hypercube $\mathbf{K} = [-1,1]^n$, described by the polynomial inequalities $1-x_1^2\ge 0 ,\ldots,1-x_n^2\ge 0$, we consider polynomials that allow a Schm\"udgen-type representation of bounded degree $r$:
\begin{equation}\label{eqSch}
p(\mathbf{x}) =  \sum_{I \subseteq [n]} \sigma_I(\mathbf{x}) \prod_{i\in I} (1-x_i^2) \in \mathcal{T}(1-x_1^2,\ldots,1-x_n^2)_{r},
\end{equation}
where the polynomials $\sigma_I$ are sum-of-squares polynomials with degree at most $r-2|I|$ (to ensure  that the degree of each term in the sum in \eqref{eq:trubcated preordering def} is at most $r$).

We call a \emph{Positivstellensatz} effective if one may bound the degree $r$ in terms of properties of $p$.
The best bound of this type is due to Slot and Laurent \cite{Laurent2021} (with an alternative proof by Bach and Rudi \cite{Bach_Rudi}).

\begin{theorem}[Laurent and Slot \cite{Laurent2021}]
\label{thm:Laurent_Slot1}
Assume $f \in \mathbb{R}[x]_d$ and $f_\min := \min_{\mathbf{x} \in [-1,1]^n} f(\mathbf{x}) > 0$. Then
\[
f \in \mathcal{T}(1-x_1^2,\ldots,1-x_n^2)_{(r+1)n}
\]
provided that
\[
r \ge \max \left\{\pi d \sqrt{2n}, \frac{1}{\sqrt{f_\min}}\sqrt{C(n,d)(f_\max - f_\min)} \right\},
\]
where $C(n,d) > 0$ depends only on $n$ and $d$, and may be chosen to be polynomial in $n$ for fixed $d$, or polynomial in $d$ for fixed $n$.
In particular
\[
C(n, d)\le  2\pi^2d^2n2^{n/2}(d + 1)^n \mbox{ and } C(n, d) \le 2\pi^2 d^2n2^{d/2}(n + 1)^d.
\]
\end{theorem}

One may formulate this result in a different way after introducing the following semidefinite programming (SDP) approximation
of $f_\min := \min_{\mathbf{x} \in [-1,1]^n} f(\mathbf{x})$:
\begin{equation}\label{Schmuedgen_SDP}
  f_{(r)} := \sup \left\{ \lambda \; | \; f - \lambda \in  \mathcal{T}(1-x_1^2,\ldots,1-x_n^2)_{r}\right\} \le f_\min.
\end{equation}
In particular, one has the following result.

\begin{theorem}[Laurent and Slot \cite{Laurent2021}]
\label{thm:Laurent_Slot2}
One has
\[
f_\min -  f_{((r+1)n)} \le  \frac{1}{r^2}C(n,d)(f_\max - f_\min) \quad \forall r \ge \pi d \sqrt{2n},
\]
where $C(n,d)$ is as in Theorem \ref{thm:Laurent_Slot1}.
\end{theorem}
	
	\subsection*{Contributions}
	In this paper we prove similar results to Theorems \ref{thm:Laurent_Slot1} and \ref{thm:Laurent_Slot2}.
We will consider a stronger SDP relaxation than \eqref{Schmuedgen_SDP}, by considering the truncated pre-ordering
\[
\mathcal{T}(1\pm x_1,\ldots,1 \pm x_n)_{r} := \mathcal{T}(1+ x_1, 1 - x_1, \ldots,1+ x_n, 1 - x_n)_{r}
\]
 instead of $\mathcal{T}(1-x_1^2,\ldots,1-x_n^2)_{r}$.
\rev{
As noted by, e.g.\, \revv{Powers and Reznick \cite[p.\ 4682]{PR2} and} Baldi and Slot \cite[\S 6]{Baldi_Slot}, these two truncated pre-orderings are closely related through the identities:
\begin{align*}
    1 \pm x_i & = \frac{1}{2} \left( (1 \pm x_i)^2 + 1 - x_i^2 \right) \\
    1 - x_i^2 & = (1-x_i)(1+x_i).
\end{align*}
In particular, the identities imply
\begin{equation}\label{eq:preorderings relation}
 \mathcal{T}(1-x_1^2,\ldots,1-x_n^2)_{r} \subset \mathcal{T}(1\pm x_1,\ldots,1 \pm x_n)_{r} \subset  \mathcal{T}(1-x_1^2,\ldots,1-x_n^2)_{2r}.
\end{equation}
Of course,  without truncation, the two pre-orderings coincide, i.e.\ $\mathcal{T}(1\pm x_1,\ldots,1 \pm x_n) = \mathcal{T}(1-x_1^2,\ldots,1-x_n^2)$.
}

Thus we consider the SDP lower bound on $f_\min$ given by
\begin{equation}\label{Schmuedgen_SDP2}
  \hat f_{(r)} := \sup \left\{ \lambda \; | \; f - \lambda \in  \mathcal{T}(1\pm x_1,\ldots,1 \pm x_n)_{r}\right\} \ge f_{(r)}.
\end{equation}
\rev{Due to the relation \eqref{eq:preorderings relation} between the two truncated pre-orderings, one has $f_{(r)} \le \hat f_{(r)} \le f_{(2r)}$.}

We will show \revv{(see Theorem \ref{thm:error_bound1} and Proposition \ref{prop:Theta_bound_nr})} that
\begin{equation}\label{eq:main result}
  f_\min -  \hat f_{(rn)} \le \|f\|_{1, T}\frac {\pi^2 d^2}{(r+2)^2} \quad \forall r \ge \pi d,
\end{equation}
where $\|f\|_{1, T}$ is the sum of the absolute values of the coefficients of $f$ in the  Chebyshev basis, i.e.\ the $1$-norm of the coefficients in
this basis.
This result is clearly of the same flavor as Theorem \ref{thm:Laurent_Slot2}, but the right-hand-side may be computed in an elementary way, and does not
explicitly depend on $n$. Moreover, the norm $\|f\|_{1,T}$ can be of the same order as $(f_\max - f_\min)$ for some classes of polynomials.
\rev{
In particular, as discussed in Section \ref{sec:norms of polynomials}, this is the case for the standard polynomial optimization reformulation of the maximum cut problem in graphs.
}
Furthermore, by comparing different norms of polynomials of bounded degrees, one may use our result to derive a result very similar to Theorem \ref{thm:Laurent_Slot2}.

We will also show an upper bound of the form $f_\min -  \hat f_{(r)} \le \Theta^{r}_{n,d}\|f\|_{1, T}$, that is valid for all $r$ (not only multiples of $n$),
 where the quantity $\Theta^{r}_{n,d}$ depends
on $(r,n,d)$ only, and may be upper bounded by solving an SDP problem;  \revv{see Theorem \ref{thm:error_bound1}}. Thus an upper bound on $\Theta^{r}_{n,d}\|f\|_{1, T}$ may readily be computed if
tables of SDP upper bounds on $\Theta^{r}_{n,d}$ are given. We give such  (partial) tables in an appendix to this paper for small values of $n$.

\revv{
\subsection*{Outline}
Our paper is structured as follows.
Section \ref{sec:preliminaries} contains notation and background material, including a review of some properties of Chebyshev polynomials (\S \ref{subsec:Chebyshev}), and an
overview of relations between some (equivalent) norms on spaces of polynomials of fixed degree (\S \ref{sec:norms of polynomials}).
We derive a key technical lemma in Section \ref{sec:key lemma} that relates the SDP hierarchy  errors $f_\min -  \hat f_{(r)}$ to the $1$-norm distance of the objective function to the truncated pre-ordering $\mathcal{T}(1\pm x_1,\ldots,1 \pm x_n)_{r}$. After we established the link with $1$-norm approximation, we explore the implications from the literature on projections of polynomials onto sum-of-squares polynomials in the $1$-norm in Section \ref{sec:1 norm projections}.
In particular, we show there how results by Lasserre and Netzer  \cite{Lasserre_Netzer_2006}, and Lasserre \cite{Lass_arxiv}, can be combined with our approach to yield insights into the convergence rate of the SDP hierarchies \eqref{Schmuedgen_SDP} and \eqref{Schmuedgen_SDP2}.
Our main results are presented in Sections \ref{sec:kernel approximation} and \ref{sec:Jackson} where we show how to obtain
convergence rates by using suitable kernels of the Jackson type. In particular, we prove the error bound \eqref{eq:main result} there in Theorem \ref{thm:error_bound1} and Proposition \ref{prop:Theta_bound_nr}. In the penultimate Section \ref{sec:SDP bounds}, we show how to bound the aforementioned quantity $\Theta^{r}_{n,d}$ numerically via the solution of an SDP problem. We conclude with some discussion on future research directions and related problems in Section \ref{sec:conclusion}.
}

	\section{Notation and preliminaries}
\label{sec:preliminaries}

\subsection{Further notation}
We will denote $\N_0 = \N \cup \{0\}$, and
\[
\mathbb{N}^n_d := \left\{ \alpha \in (\mathbb{N}_0)^n \; | \; |\alpha| := \sum_{i=1}^n \alpha_i \le d\right\},
\]
and $\mathbf{x}^\alpha = x_1^{\alpha_1}\ldots x_n^{\alpha_n}$.
\rev{We will be considering polynomials of different types of variables, and will need to keep track of the degree of these polynomials with respect to only some of the variables. To do this, given a multivariate polynomial $p$ and a subset of the variables $S$, we define $\deg_S p$ as the \emph{degree of $p$ with respect to the variables in $S$}. That is, when all the variables not in $S$ are considered to be constants (or more correctly parameters). For instance let $p(\varu,\x,\y) = \varu^\alpha\x^\beta\y^\gamma$, then $\deg p = |\alpha|+|\beta| + |\gamma|$ while $\deg_\x p = |\beta|$ and $\deg_{\varu\x} p = |\alpha|+|\beta|$.}

If we consider functions of two distinct types of variables on the hypercube, say $(\mathbf{x},\mathbf{y})\in [-1,1]^n\times [-1,1]^n$, then
we define an associated truncated pre-ordering as
\rev{
\begin{eqnarray*}
   &&\mathcal{T}(1\pm x_1,\ldots,1 \pm x_n; 1\pm y_1,\ldots,1 \pm y_n)_{r_1,r_2} \\
   &&\qquad:= \left\{\sum_{I^-_\x,I^+_\x,I^-_\y,I^+_\y \subseteq [n]}
\sigma_{I^-_\x,I^+_\x,I^-_\y,I^+_\y }(\x,\y)\prod_{i \in I^-_\x}(1- x_i)\prod_{i \in I^+_\x}(1+ x_i)\prod_{i \in I^-_\y}(1- y_i)\prod_{i \in I^+_\y}(1+ y_i) \right\}, \\
  \end{eqnarray*}
where, for all $I^-_\x,I^+_\x,I^-_\y,I^+_\y \subseteq [n]$, one has $ \sigma_{I^-_\x,I^+_\x,I^-_\y,I^+_\y }\in \Sigma[\mathbf{x},\mathbf{y}]$, as well as
$
|I^-_\x|+ |I^+_\x|+ \mbox{deg}_\x \sigma_{I^-_\x,I^+_\x,I^-_\y,I^+_\y } \le r_1$, and
$|I^-_\y| +|I^+_\y| + \mbox{deg}_\y \sigma_{I^-_\x,I^+_\x,I^-_\y,I^+_\y } \le r_2$.
In other words, in this special truncated pre-ordering, the degree of each term is at most $r_1$ in the $\mathbf{x}$ variables, and at most $r_2$ in the $\mathbf{y}$ variables.}

	\subsection{Chebyshev polynomials}\label{subsec:Chebyshev}
\rev{In this section we review known results on Chebyshev polynomials for later use. The interested reader may find more information in the classic text by Rivlin \cite{Riv}.}
	
	Let $\mathbf{K} = [-1,1]$  and fix the measure $\mu$ on $\mathbf{K}$ defined by $\mathrm{d}\mu(x) = (\pi \sqrt{1-x^2})^{-1}\mathrm{d}x$, $x \in \mathbf{K}$. The Chebyshev polynomials of the first kind form a system of orthogonal polynomials \rev{for  $\mu$}. We will refer to the $k$-th Chebyshev polynomial of first kind as $T_k(x)$. We have, for $k \in \N_0$,
	\begin{align}
	T_{k}(x) & = \cos (k \arccos x ).
	\end{align}

	Define for $f,g : [-1,1] \rightarrow \mathbb{R}$
	\[ \langle f,g \rangle_{\mu} = \int_{-1}^1 \frac{f(x)g(x)}{\pi \sqrt{1-x^2} } \mathrm{d}x \]
	to obtain the following orthogonality relations for the Chebyshev polynomials of the first kind
	\begin{align}
	\langle T_k, T_m \rangle_{\mu} &= \frac{1+\delta_{k,0}}{2} \delta_{k,m}.
	\end{align}
We will denote the normalized Chebyshev polynomials by
\[
\hat T_0 = T_0, \; \hat T_k = \sqrt{2}T_k \; (k \in \mathbb{N}).
\]

	 It is straightforward to generalize the Chebyshev polynomials to the multivariate case by taking
products of univariate Chebyshev polynomials. Let $\mathbf{K} = [-1,1]^n$ and define
	\begin{equation}\label{mu cheb}
	  \mathrm{d}\mu(\textbf{x}) := \prod_{i=1}^{n} \frac{1}{\pi \sqrt{1-x_i^2}}\mathrm{d}\textbf{x}.
	\end{equation}

	Then, for $\alpha \in (\mathbb{N}_0)^n$ the corresponding (normalized) multivariate Chebyshev polynomial of the first kind is defined as
	\[
	\hat T_\alpha(\textbf{x}) = \prod_{i=1}^n \hat T_{\alpha_i}(x_i).
	\]
	
	The orthogonality relations extend in the following way
	
	\begin{equation*}
		\begin{aligned}
		\langle \hat T_\alpha, \hat T_\beta \rangle_{\mu} = \int_{\mathbf{K}} \hat T_\alpha(\textbf{x})\hat T_\beta(\textbf{x}) \mathrm{d}\mu(\textbf{x})
		 & = \prod_{i=1}^{n} \int_{-1}^{1} \frac{\hat T_{\alpha_i}(x_i) \hat T_{\beta_i}(x_i)}{\pi \sqrt{1-x_i^2}}\mathrm{d}x_i = \prod_{i= 1}^{n} \delta_{\alpha_i,\beta_i}.
		\end{aligned}
	\end{equation*}
 On has the following formula to convert from the univariate monomial basis to the Chebyshev basis:
\begin{equation}
x^k=2^{1-k}{\mathop{{\sum}'}_{\overset{j=0}{k-j \mbox{ even}}}^k}\binom{k}{(k-j)/2} T_j(x), \quad (k \in \N),
\label{eq.xnofT}
\end{equation}
where the prime at the sum symbol means the first term (at $j=0$ for even $n$) should be halved.
	
\subsection{Norms of polynomials}
\label{sec:norms of polynomials}
The $n$-variate polynomials of degree at most $d$, denoted by $\mathbb{R}[\mathbf{x}]_d$, may be
 associated with $\mathbb{R}^{{n+d \choose d}}$ via the polynomial coefficients in a
given basis.

Using the standard monomial basis, any $f \in \mathbb{R}[\mathbf{x}]_d$ may be written as
$f(\mathbf{x}) = \sum_{\alpha \in \N^n_d} f_\alpha \mathbf{x}^\alpha$, with associated  $1$-norm:
\[
\|f\|_1 := \sum_{\alpha \in \N^n_d} |f_\alpha|.
\]
Since the $1$-norm depends on the basis, we will use the notation  $\|\cdot\|_{1,T}$ (resp.\ $\|\cdot\|_{1,\hat T}$)   for the  Chebyshev (resp.\
 normalized Chebyshev) basis. Note that $\|f \|_{1,\hat T} \le \|f\|_{1, T} \le 2^{d/2}\|f\|_{1,\hat T}$ if $f \in \mathbb{R}[x]_d$ and
 $\|f \|_{1,\hat T} = \frac{1}{\sqrt{2}^{d}} \|f\|_{1, T}$ if $f$ is homogeneous of degree $d$.

\begin{lemma}
One has
 \begin{equation}\label{one_norm_conversion}
   \|f\|_{1, T} \le \|f\|_1 \quad \forall \; f \in \R[\x].
 \end{equation}
\end{lemma}
\begin{proof}
 Since $T_j(1) = 1$ for all $j \in \N_0$,
 all the coefficients on the right-hand-side of
 \eqref{eq.xnofT} are positive and sum to one. This implies that, for any (multivariate) monomial $f(\x) = \x^\alpha$ with $\alpha \in \N^n$, one has
 $ \|f\|_{1, T} = \|f\|_1 = 1$. Thus, if $f(\x) = \sum_{\alpha \in \N^n_d} f_\alpha \x^\alpha$ is a general degree $d$ polynomial, one has (with slight abuse of notation):
 \[
   \|f\|_1 = \sum_{\alpha \in \N^n_d} |f_\alpha|
    =  \sum_{\alpha \in \N^n_d} \|f_\alpha \x^\alpha\|_{1,T}
    \ge   \left\|\sum_{\alpha \in \N^n_d} f_\alpha \x^\alpha\right\|_{1,T} \equiv \|f\|_{1,T},
 \]
 where the inequality is the triangle inequality.
\end{proof}

 We will only be concerned with the $1$-norm for different bases, but in general one may define the $p$-norm for any $p >0$:
 \[
\|f\|_p := \left(\sum_{\alpha \in \N^n_d}  |f_\alpha|^p\right)^{1/p}.
\]
 To avoid confusion with the usual $\ell_p$ norm, we will use the notation, for any $p \in \mathbb{N}$,
 \[
 \|f\|_{\ell_p} := \left(\int_{\mathbf{K}} |f(\mathbf{x})|^p d\mathbf{x}\right)^{1/p},
 \]
where we only consider the case $\mathbf{K} = [-1,1]^n$,
and the sup-norm:
 \[
 \|f\|_{\ell_\infty} := \sup_{\mathbf{K}} |f(\mathbf{x})|.
 \]
The relations between the (equivalent) norms $\|\cdot\|_p$ and $\|\cdot \|_{\ell_p}$ on $\mathbb{R}[x]_d$ is a classical problem that is still a topic of active research; see, e.g.\ \cite{Equiv_norms} for a recent overview.
Perhaps the best-known result along these lines is the Bohnenblust-Hille inequality:
\begin{equation}\label{Bohnenblust-Hille}
   \|f\|_p \le C_{d}\|f\|_{\ell_\infty} \quad \mbox{deg}(f) = d, \; p = \frac{2d}{d+1},
\end{equation}
where the constant $C_{d}$ depends on $d$ but not on $n$.
For values of $p <   \frac{2d}{d+1}$, the corresponding constant does depend on $n$.
One may obtain a bound for the $1$-norm instead by using the known inequalities between $p$-norms: if $0 < q < p$ one has
\[
\|\mathbf{x}\|_p \le \|\mathbf{x}\|_q \le k^{1/q - 1/p}\|\x\|_p \quad \forall \mathbf{x} \in \mathbb{R}^k.
\]
Setting $q=1$, $p  = \frac{2d}{d+1}$, and $k = {n+d \choose d}$, \eqref{Bohnenblust-Hille} implies
\begin{equation}\label{one_vs_sup_norm}
    \|f\|_1 \le {n+d \choose d}^{1 - \frac{d+1}{2d}}C_{d}\|f\|_{\ell_\infty} = {n+d \choose d}^{ \frac{d-1}{2d}}C_{d}\|f\|_{\ell_\infty}
    < {n+d \choose d}^{ \frac{1}{2}}C_{d}\|f\|_{\ell_\infty},
\end{equation}
where $C_d$ is the same constant as in \eqref{Bohnenblust-Hille}.

Some stronger results are known in special cases, e.g, when $n=2$ and $d=2$, one has
\[
\|f\|_{\ell_\infty} \le \|f\|_1 \le (1+\sqrt{2})\|f\|_{\ell_\infty}.
\]
As another example, if $f(x) = \frac{1}{4}x^TLx$ where $L$ is the Laplacian matrix of a graph, then $\|f\|_{\ell_\infty}$ is the maximum cut size in the graph, while $\|f\|_1$ equals the number of edges in the graph. Thus, in this case one has $\|f\|_{\ell_\infty} \le \|f\|_1 \le 2\|f\|_{\ell_\infty}$, i.e.\ the constant in the upper bound ($2$ in this case) need not depend on $n$.
	
\section{A key lemma}
\label{sec:key lemma}

As before, let $T_k$ denote the univariate Chebyshev polynomials of the first kind.
Then one has the well-known recursive relation $T_0(t) = 1$,  $T_1(t) = t$ and $T_{k+1}(t) = 2t T_k (t)- T_{k-1}(t)$.
Recall that we define the multivariate version by $T_{\alpha} (\mathbf{x})= \prod_{i=1}^n T_{\alpha_i}(x_i)$.
\rev{
Recall that $1 \pm  T_{\alpha}$ is nonnegative on $[-1,1]^n$.
Our first goal is to prove the stronger result that $1 \pm  T_{\alpha}$ belongs to the truncated pre-ordering of order $|\alpha|$, namely
$1 \pm T_\alpha(\mathbf{x}) \in \mathcal{T}(1-x_1,1+x_1,...,1-x_n,1+x_n)_{|\alpha|}$.

In the proof, we will need the a classical result of this type for the special case of univariate polynomials.
\begin{theorem}[Fekete, Markov-Luk\'acz (see, e.g., \cite{PR2})]\label{THM:MarkovLukacz}
Let $p$ be a univariate polynomial of degree
$2m$. Then $p$ is nonnegative on the interval $[a,b]$ if and
only if
\[
 p(t)=q_1^2(t)+(t-a)(t-x)q_2^2(t),
\]
for some polynomials $q_1$ of degree $m$ and $q_2$ of degree $m-1$.
If the degree of $p$ is $2m+1$ then $p$ is
nonnegative on $[a,b]$ if and only if
\[
 p(t)=(t-a)q_1^2(t)+(b-t)q_2^2(t),
\]
for some polynomials $q_1$ and $q_2$ of degree $m$.
\end{theorem}
We now give a multivariate generalization of this result.
}

\begin{lemma}
For any $\alpha \in (\mathbb{N}_0)^n$  we have $1 \pm T_\alpha(\mathbf{x}) \in \mathcal{T}(1\pm x_1,\ldots,1\pm x_n)_{|\alpha|}$.
\end{lemma}
\proof
 By applying the identities,
\[
1+ab = [(1-a)(1+b) + (1+a)(1-b)]/2 \mbox{ and } 1-ab = [(1+a)(1+b) + (1-a)(1-b)]/2
\]
repeatedly as necessary, we obtain that  $1 \pm T_{\alpha}(\mathbf{x}) = 1 \pm \prod_{i=1}^n T_{\alpha_i}(x_i)$
 can be expressed as sums of products of terms of the form $1 \pm T_{\alpha_i}(x_i)$.

Thus, it is enough to show that   $1 \pm T_k(t) \in \mathcal{T}(1-t,1+t)_k$.
But this follows from the fact that every $p \in \R[t]_k$  non-negative  on $[-1,1]$ is in  $\mathcal{T}(1-t,1+t)_k$, \rev{by Theorem \ref{THM:MarkovLukacz},} and the fact that
and $1 \pm T_k$ is nonnegative on $[-1,1]$.   \qed

\rev{The main result of this section now follows as an immediate consequence of the last lemma.}
\begin{theorem}\label{thm:norm1}
For all $p \in \R[\mathbf{x}]$, one has $\|p\|_{1,T} - p \in \mathcal{T}(1 \pm x_1,...,1\pm x_n)_{\deg p}$.
\end{theorem}
\proof
Letting $p(\mathbf{x}) = \sum_{\alpha \in \N^n_d} p_\alpha T_\alpha(\mathbf{x})$, we
  get $\|p\|_{1,T} - p = \sum_{\alpha \in \N^n_d} |p_\alpha| (1-\mbox{sign}(p_\alpha) T_\alpha(\mathbf{x}))$ and thus the statement follows by the last lemma.\qed

\begin{corollary}
\label{cor:Juan}
If $f$, $p \in \R[\mathbf{x}]_d$, and $p \in \mathcal{T}(1\pm x_1,\ldots,1\pm x_n)_d$,  then
\[
f + \|p-f\|_{1,T} \in \mathcal{T}(1\pm x_1,\ldots,1\pm x_n)_d.
\]
\rev{
In particular, replacing $f$ by $f - f_{\min}$, one has
\begin{equation}\label{fr bound ito one norm}
  \hat f_{(r)} \ge f_{\min} -  \min_{p \in \mathcal{T}(1\pm x_1,\ldots,1\pm x_n)_r} \|p - (f - f_{\min})\|_{1,T},
\end{equation}
for $r \ge \deg f$, where $\hat f_{(r)}$ is the SDP bound defined in \eqref{Schmuedgen_SDP2}.
}
\end{corollary}

\textbf{{Remark:} }these results remain true if the Chebyshev basis is replaced by the standard monomial basis, with our prior convention that, if
$f(\mathbf{x}) = \sum_{\alpha \in \N^n_d} f_\alpha \mathbf{x}^\alpha$, then $f$ has $1$-norm $\|f\|_1 = \sum_{\alpha \in \N^n_d} |f_\alpha|$.

\section{A quantitatative Schm\"{u}dgen-type theorem for the hypercube via sums-of-squares approximations}
\label{sec:1 norm projections}
We can use Corollary \ref{cor:Juan} to derive a quantitative version of Schm\"{u}dgen-type theorem for the $[-1,1]^n$ hypercube.
We will need the following result from Lasserre and Netzer  \cite{Lasserre_Netzer_2006}.

\begin{theorem}[Lasserre and Netzer  \cite{Lasserre_Netzer_2006}]
\label{lemma:Netzer}
If $f \in \mathbb{R}[\mathbf{x}]$ is nonnegative on $[-1,1]^n$, then, for any $\epsilon > 0$, there exists an $r_\epsilon \in \mathbb{N}$ such that
\begin{equation}
\label{eq:sos-perturbation}
f_\epsilon(\mathbf{x}) := f(\mathbf{x}) + \epsilon\left(1 + \sum_{i=1}^n x_i^{2r_\epsilon}\right) \in \Sigma[\x]_{2r_\epsilon}.
\end{equation}
\end{theorem}

We now have the following Schm\"{u}dgen-type theorem.
\begin{theorem}
If $f \in \mathbb{R}[\mathbf{x}]$ is nonnegative on $[-1,1]^n$, then, for any $\epsilon > 0$, there exists an $r \in \mathbb{N}$ such that
\[
f + (n+1)\epsilon \in \mathcal{T}(1-x_1,1+x_1,...,1-x_n,1+x_n)_{2r}.
\]
\revv{The value $r$} is upper bounded by the smallest integer $r_\epsilon$ such that \eqref{eq:sos-perturbation} holds.
\end{theorem}
\textit{Proof.}
Note that $\|f - f_\epsilon\|_1 = (n+1)\epsilon$ in the standard monomial basis.
The result now follows from Corollary \ref{cor:Juan}. \qed

Lasserre
\cite{Lass_arxiv}
characterized the best sums-of-squares approximation in $1$-norm.

\begin{theorem}[Lasserre \cite{Lass_arxiv}]
Let $f \in \mathbb{R}[\mathbf{x}]$, and $d \in \mathbb{N}$ such that $\deg f \le 2d$. Then there exist nonnegative
$\lambda^*_0, \ldots, \lambda_n^*$, such that
\[
\mathbf{x} \mapsto f(\mathbf{x}) + \lambda_0^* + \sum_{i=1}^n \lambda^*_i x_i^{2d} \in \arg\min_{p \in \Sigma[\mathbf{x}]_{2d}} \|f - p\|_1.
\]
Moreover, the values $\lambda^*_0, \ldots, \lambda_n^*$ are given by the optimal solution of the SDP
\[
\rho_d := \min_{\lambda \ge 0} \left\{ \sum_{i=0}^n \lambda_i \; : \; \x \mapsto f(\mathbf{x}) +
 \lambda_0 + \sum_{i=1}^n \lambda_i x_i^{2d} \in \Sigma[\mathbf{x}]_{2d}\right\},
\]
i.e.
\begin{equation}
\label{eq:rho_d}
\rho_d = \sum_{i=0}^n \lambda^*_i = \min_{p \in \Sigma[\mathbf{x}]_{2d}} \|f - p\|_1.
\end{equation}
\end{theorem}

\rev{
By Corollary \ref{cor:Juan}, we have the following immediate implication of the last theorem.}
\begin{corollary}
\label{cor:rho_d}
    If $f$ has $\deg f < 2d$,   then
\[
f + \rho_d \in \mathcal{T}(1-x_1,1+x_1,\ldots,1-x_n,1+x_n)_{2d},
\]
where $\rho_d$ is defined in \eqref{eq:rho_d}.
\end{corollary}
\rev{
The corollary implies that, if $\deg f < 2d$ and $r \ge d$, one has $\hat f_{(2r)} \ge  - \rho_r$.}

\begin{example}[Lasserre \cite{Lass_arxiv}]
Consider the Motzkin-type polynomial  $\x\mapsto f(\mathbf{x})=x_1^2x_2^2(x_1^2+x_2^2-1)+1/27$ of degree $6$, which is nonnegative but not
a sum of squares, and with a global minimum $f_\min=0$ on $[-1,1]^2$, attained at  four global minimizers $x^*=(\pm (1/3)^{1/2}, \pm (1/3)^{1/2})$.
The values $\rho_d$  are displayed in Table \ref{tab1} for $d=3,4,5$.
\begin{table}[h!]
\begin{center}
\begin{tabular}{|| l | l | l ||}
\hline
$d$ & $\lambda^*$  & $\rho_d$ \\
\hline
\hline
$3$ &  $\approx 10^{-3}\,(5.445,  5.367 , 5.367)$ & $\approx 1.6\, 10^{-2}$\\
$4$& $\approx 10^{-4}\,(2.4 ,  9.36 , 9.36)$ &$\approx 2.\,10^{-3}$\\
$5$& $\approx 10^{-5}\,(0.04 ,  4.34, 4.34)$ &$\approx 8.\,10^{-5}$\\
\hline
\end{tabular}
\end{center}
\caption{\label{tab1} Best $1$-norm approximation for the Motzkin polynomial.}
\end{table}
For this example one has $f \in \mathcal{T}(1-x_1,1+x_1,1-x_2,1+x_2)_{6}$ so that
the lower bounds $-\rho_d$  on $f_\min = 0$ via Corollary \ref{cor:rho_d} are not tight for $d = 3,4,5$.
\end{example}

\section{$1$-norm appoximation by kernels}
\label{sec:kernel approximation}
\rev{Corollary \ref{cor:Juan} showed that one may bound the errors in the Schm\"{u}dgen SDP hierarchy \eqref{Schmuedgen_SDP2} in terms
of suitable $1$-norm approximations of the objective function. In this section we therefore discuss a systematic approach to $1$-norm approximation of polynomials via positive kernels. This is a classical approach in approximation theory, see e.g.\ \cite{Altomare2011}, and the survey \cite{Weisse2006}.}

 We now consider a kernel $K_r : \mathbb{R}^n \times \mathbb{R}^n \rightarrow \mathbb{R}$ given by
	\[
	K_r(\textbf{x},\textbf{y}) := \sum_{\alpha \in \mathbb{N}^n_r} \lambda_\alpha \hat T_\alpha(\textbf{x}) \hat T_\alpha(\textbf{y}),
	\]
	for given constants $\lambda_\alpha \in [0,1]$ for $\alpha \in \mathbb{N}^n_r$.
	The convolution operator:
	\begin{equation}\label{eq:convolution}
	  \mathcal{K}^{(r)}(f)(\textbf{x}) := \int_{\mathbf{K}} f(\textbf{y})K_r(\textbf{x},\textbf{y})\mathrm{d}\mu(\textbf{y}),
	\end{equation}
		maps any  integrable function $f$ to a polynomial of degree at most $r$. More precisely,
	\[
	\mathcal{K}^{(r)}(f)(\textbf{x})  = \sum_{\alpha \in \mathbb{N}^n_r} b_\alpha \hat T_\alpha(\textbf{x}) \,,
	\mbox{ where }
	b_\alpha = \langle \hat T_\alpha, f \rangle_\mu\, \lambda_\alpha.
	\]
	The coefficients $\lambda_\alpha$ of the kernel $K_r$ determine the approximation.
	In particular, if $f(\mathbf{x}) = \sum_{\alpha \in \mathbb{N}^n_d} f_\alpha  T_\alpha(\mathbf{x})$ is a polynomial of degree $d\le r$, then

	\[
	\mathcal{K}^{(r)}(f)(\textbf{x})  = \sum_{\alpha \in \mathbb{N}^n_d} \lambda_\alpha f_\alpha T_\alpha(\textbf{x}) \,,
		\]
	so that the operator $\mathcal{K}^{(r)}$ is a  projection if all $\lambda_\alpha = 1$ (the so-called Dirichlet  or Christoffel-Darboux kernel, see e.g.\ \cite{Christoffel-Darboux}).
	\rev{If $\lambda_{\bf 0} =1$, then the convolution operator preserves constant functions. We will assume this property throughout.}

Also note that, if $f(\mathbf{x}) = \sum_{\alpha \in \mathbb{N}^n_d} f_\alpha  T_\alpha(\mathbf{x})$, and $r \ge d$,
\[
\| \mathcal{K}^{(r)}(f) - f\|_{1, T}  = \sum_{\alpha \in \mathbb{N}^n_d} |1- \lambda_\alpha||f_\alpha| \le \max_{\alpha \in \mathbb{N}^n_d}|1- \lambda_\alpha| \cdot \|f\|_{1, T}.
\]

We are interested in positive kernels that map nonnegative functions on $\mathbf{K}=[-1,1]^n$ to nonnegative functions, i.e.\ positive linear
approximation operators in the sense of Korovkin; see e.g.\ \cite{Altomare2011}.
More precisely, \rev{for given degree $d \in \mathbb{N}$, we are interested in kernels where the coefficients $\lambda_\alpha$ are such that
\begin{equation}\label{kernel_preorder_condition}
  \mathcal{K}^{(r)}(f)(\textbf{x}) := \int_{\mathbf{K}} f(\textbf{y})K_r(\textbf{x},\textbf{y})\mathrm{d}\mu(\textbf{y}) \in \mathcal{T}(1\pm x_1,\ldots, 1 \pm x_n)_r \quad \forall \mbox{ $f \in \mathbb{R}[\x]_d$ nonnegative on $\mathbf{K}$.}
\end{equation}
}
Note that the Dirichlet kernel is not positive; see e.g.\ \cite{Weisse2006}. However, we would like the coefficients $\lambda_\alpha$ as close to
  $1$ as possible, to approximate the Dirichlet kernel, since this kernel gives
  a spectral rate of convergence for analytic functions; see e.g.\ \cite{Trefethen2017}.

This leads us naturally to the following abstract optimization problem, that aims to construct the best kernel that meets our requirements.
\rev{
\begin{equation}\label{SDP:kappa_r}
  \Theta^{r}_{n,d} := \inf \left\{ \max_{\alpha \in \mathbb{N}^n_d}|1- \lambda_\alpha| \; : \;
K_r(\textbf{x},\textbf{y})  = 1+ \sum_{ \alpha \in {\mathbb{N}}^n_r\setminus\{\mathbf{0}\}} \lambda_\alpha \hat T_\alpha(\textbf{x})\hat T_\alpha(\textbf{y}) \mbox{ satisfies } \eqref{kernel_preorder_condition}\right\}.
\end{equation}
}

One may bound the optimal value $\hat f_{(r)}$ of the problem \eqref{Schmuedgen_SDP2} in terms of $\Theta^{r}_{n,d}$ as follows.

\begin{theorem}
\label{thm:error_bound1}
Let $f \in \mathbb{R}[\mathbf{x}]_d$ and $f_\min = \min_{\mathbf{x} \in \mathbf{K}} f(\mathbf{x})$.
One has, for each $r \in \mathbb{N}$ with $r \ge d$,
\[
f - (f_\min - \Theta^{r}_{n,d}\|f\|_{1, T}) \in \mathcal{T}(1\pm x_1,\ldots, 1 \pm x_n)_{r}
\]
which implies $f_\min - \hat f_{(r)}  \le \Theta^{r}_{n,d}\|f\|_{1, T}$, where $\hat f_{(r)}$ is the SDP bound from \eqref{Schmuedgen_SDP2}.
\end{theorem}
\proof
If $\lambda_\alpha$ ($\alpha \in {\mathbb{N}}^n_r$) denotes \revv{a feasible solution of problem \eqref{SDP:kappa_r}, one has
\[
\| \mathcal{K}^{(r)}(f) - f\|_{1, T}  = \sum_{\alpha \in \mathbb{N}^n_d} |1- \lambda_\alpha||f_\alpha| \le \max_{\alpha \in {\mathbb{N}}^n_r} |1- \lambda_\alpha| \|f\|_{1,T}.
\]
}
\rev{Also, with reference to \eqref{fr bound ito one norm},  one has:
\[
\mathcal{K}^{(r)}(f-f_\min) - (f-f_\min) = \mathcal{K}^{(r)}(f)-f_\min - (f-f_\min)= \mathcal{K}^{(r)}(f) - f,
\]
since $\mathcal{K}^{(r)}$ preserves constant functions.
Finally, by assumption one has $  \mathcal{K}^{(r)}(f-f_\min)  \in \mathcal{T}(1\pm x_1,\ldots, 1 \pm x_n)_{r}$, so that the required result follows from
\eqref{fr bound ito one norm}.}\qed

Some comments on the theorem:
\begin{itemize}
\item
The values $\Theta^{r}_{n,d}$ do not depend on $f$, and may be upper-bounded a priori via SDP, as we will show later. Thus the right-hand-side
of the inequality
$f_\min - \hat f_{(r)}  \le \Theta^{r}_{n,d}\|f\|_{1, T}$ may readily be upper bounded, given  a table of SDP upper bounds on the values $\Theta^{r}_{n,d}$.
\item
We will show in the next section that $\Theta^{r}_{n,d} = O(1/r^2)$ by considering a specific kernel that is feasible for problem \eqref{SDP:kappa_r}. This is essentially the same
argument used by Laurent and Slot \cite{Laurent2021}. In particular, it suffices to use products of so-called {\em univariate Jackson kernels}.
\end{itemize}

\section{Kernels of the Jackson type}
\label{sec:Jackson}
For $r\in \N$, consider the univariate kernel $\jackpoly_r : \R \times \R \to \R$ given by:
\begin{equation}
    \label{EQ:univariatejackpoly}
    \jackpoly_r(x, y) := 1 +  \sum_{k=1}^r \jacksoncoef{k}{r} \hat T_{k}(x) \hat T_{k}(y),
\end{equation}
where the coefficients $\jacksoncoef{k}{r}$ are given by
\begin{equation}
    \label{EQ:jacksoncoef}
    \jacksoncoef{k}{r} = \frac{1}{r+2}\big((r+2-k) \cos(k \theta_r) + \frac{\sin(k\theta_r)}{\sin(\theta_r)} \cos(\theta_r) \big) \quad (1 \leq k \leq r),
\end{equation}
with $\theta_r = \frac{\pi}{r+2}$.

This kernel is called the {\em univariate Jackson kernel} in \cite{Weisse2006} and \cite{Laurent2021}. In the approximation theory literature there is another object which is known by that name,
that corresponds to a kernel in the original analysis of Jackson \cite{Jackson1911,Jackson1912}, where he showed that the famous approximation
bounds by Bernstein \cite{Bernstein1912} were tight.
 In \cite{Kirchsner2023}, the kernel \eqref{EQ:univariatejackpoly} is called the \emph{minimum resolution kernel},
  since it minimizes the resolution over all positive univariate kernels of a given degree, where the resolution is defined by
  $$\sigma_r := \left(\int_{[-1,1] \times [-1,1]} ({x}-{y})^2 K_r({x},{y}) \mathrm{d}\mu({x})\mathrm{d}\mu({y})\right)^{1/2}.$$
  In this paper we will also speak about the univariate Jackson kernel, with the caveat that the term is ambiguous. The interested reader may
  consult  \cite{Kirchsner2023} for more details on the history and terminology.

The following properties of the univariate Jackson kernel were shown in \cite{Laurent2021}; see also \cite{Weisse2006} and the references therein.
\begin{proposition}[Proposition 6 in Laurent and Slot \cite{Laurent2021}]
\label{PROP:jacksoncoefficients}
For every $r \in \N$  we have:
\begin{enumerate}
	\item[(i)] $\jackpoly_r(x, y) \geq 0$ for all $x, y \in [-1, 1]$,
	\item[(ii)] $1 \geq \jacksoncoef{k}{r} > 0$ for all $1 \leq k \leq r$,  and
    \item[(iii)]  $|1 - \jacksoncoef{k}{r}| = 1 - \jacksoncoef{k}{r} \leq {\pi^2 k^2\over (r+2)^2}$ for all $1 \leq k \leq r$.
\end{enumerate}
\end{proposition}

One may obtain a multivariate kernel by multiplying univariate Jackson kernels in the obvious way:
\begin{equation}
	\label{EQ:multijackpoly}
	\multijackpoly_{nr}(\x, \y) := \prod_{i = 1}^n \jackpoly_r(x_i, y_i),
\end{equation}
where $\jackpoly_r$ is the univariate Jackson kernel from $\eqref{EQ:univariatejackpoly}$.
The coefficients of this kernel are given by
\[
\jacksoncoef{\alpha}{r} := \prod_{i = 1}^n \jacksoncoef{\alpha_i}{r} \quad (\alpha \in \N^n_r).
\]

The following is a refinement of a lemma by Laurent and Slot \cite[Lemma 12]{Laurent2021}.

\begin{lemma}\label{lemlambda}
Let $\pi d < r+2$. For any $\alpha \in \N^n_d$  we have
$|1-\lambda^r_\alpha| \le \tfrac {\pi^2 d^2}{(r+2)^2}$.
\end{lemma}
\begin{proof}
From Proposition \ref{PROP:jacksoncoefficients} part (iii), we know that $1 \ge \lambda^r_k \ge 1 - \frac {\pi^2 k^2}{(r+2)^2}\ge 0$. Then,
\[1 \ge \lambda^r_\alpha \ge \prod_{i=1}^n \left( 1 - \tfrac {\pi^2 (\alpha_i)^2}{(r+2)^2} \right) \ge 1 - \tfrac {\pi^2}{(r+2)^2} \sum_{i=1}^n (\alpha_i)^2 \ge 1 - \tfrac {\pi^2 d^2}{(r+2)^2},\]
where the third inequality is the generalised Bernoulli inequality:
\[
\prod_{i=1}^n \left( 1 - y_i \right) \ge 1 - \sum_{i=1}^n y_i \quad \mbox{ if $y_i < 1 \; \forall i \in [n]$}.
\]
Note that the condition $y_i <1$ for all $i$ in the Bernoulli inequality corresponds to the condition $\pi d < r+2$ in the statement of the lemma,
since $\alpha \in \N^n_d$.
\end{proof}

We obtain the following upper bound on $\Theta^{rn}_{n,d}$.
\begin{proposition}
\label{prop:Theta_bound_nr}
$\Theta^{rn}_{n,d} \le \frac {\pi^2 d^2}{(r+2)^2}$ for all $r$ and $d$ such that $\pi d < r+2$.
\end{proposition}
\begin{proof}
The result follows from the definition of $\Theta^{r}_{n,d}$ in \eqref{SDP:kappa_r} and Lemma \ref{lemlambda}, provided that
the multivariate Jackson kernel \eqref{EQ:multijackpoly} satisfies the property \eqref{kernel_preorder_condition}.
This indeed holds, as is shown in \cite[Lemma 9]{Laurent2021}.
\end{proof}
\rev{
We will need the argument from \cite[Lemma 9]{Laurent2021} once more in the next section in a more general context, and there we will give a self-contained proof
for the sake of completeness; see Proposition \ref{prop:cubature}.}

\section{SDP bounds on $\Theta^{r}_{n,d}$}
\label{sec:SDP bounds}
In this section we will prove that the following inequality hods.

\begin{theorem}
\label{proposition:Theta_SDP_bound}
The value $\Theta_{n,d}^r$ as defined in \eqref{SDP:kappa_r} is upper bounded as follows:
\begin{equation}\label{SDP_bound_Theta}
  \Theta_{n,d}^r \le \inf\left\{\max\{|1-p_{{\alpha}}|:|{\alpha}| \le d\}: \mathbf{x} \mapsto 1+ \sum_{\alpha \in \N^n_d\setminus \{\mathbf{0}\}}  2^{\omega(\alpha)} p_{{\alpha}}  T_{\alpha}(\mathbf{x})  \in \mathcal{T}(1\pm x_1,\dots,1\pm x_n)_{r}\right\},
\end{equation}
where $\omega(\alpha)$ denotes the number of nonzero elements of $\alpha$.
\end{theorem}
The right-hand-side of the inequality \eqref{SDP_bound_Theta} may readily be computed via SDP for given (small) values of $n,r,d$; we list the resulting numerical bounds for various values of $n,r,d$ in an appendix to this paper. \rev{In Figure \ref{fig:plots} we plot some of the values of the SDP bound, for illustrative purposes.}

\begin{figure}[h!]
\centering
\begin{subfigure}[t]{0.45\textwidth}
\centering
\begin{tikzpicture}
\begin{axis}[
    xlabel={$r$},
    ylabel={},
    grid=both,
    xmin=2, xmax=18,
    ymin=0, ymax=1,
    width=8cm,
    height=6cm,
   legend pos=south west,
    xmode=log, 
    ymode=log, 
]

\addplot[
    color=blue,
    mark=*,
    mark options={solid},
    ]
    coordinates {
        (2, 0.5556)(3, 0.5000)(4, 0.3319)(5, 0.2929)(6, 0.2159)(7, 0.1910)(8, 0.1502)(9, 0.1340)(10, 0.1100)(11, 0.0990)(12, 0.0838)(13, 0.0761)(14, 0.0659)(15, 0.0603)(16, 0.0552)(17, 0.0539)(18, 0.0508)
    };
\addlegendentry{$n=1$}

\addplot[
    color=red,
    mark=square*,
    mark options={solid},
    ]
    coordinates {
        (2, 0.7647)(3, 0.7500)(4, 0.5295)(5, 0.5000)(6, 0.3720)(7, 0.3455)(8, 0.2706)(9, 0.2500)(10, 0.2037)(11, 0.1883)(12, 0.1581)(13, 0.1464)(14, 0.1259)(15, 0.1170)(16, 0.1024)(17, 0.0955)(18, 0.0877)
    };
\addlegendentry{$n=2$}

\addplot[
    color=green,
    mark=triangle*,
    mark options={solid},
    ]
    coordinates {
        (2, 0.8400)(3, 0.8333)(4, 0.6331)(5, 0.6092)(6, 0.4951)(7, 0.4727)(8, 0.3804)(9, 0.3582)(10, 0.2961)(11, 0.2787)(12, 0.2348)(13, 0.2206)(14, 0.1897)(15, 0.1785)(16, 0.1559)
    };
\addlegendentry{$n=3$}

\addplot[
    color=orange,
    mark=diamond*,
    mark options={solid},
    ]
    coordinates {
        (2, 0.8788)(3, 0.8750)(4, 0.6943)(5, 0.6770)(6, 0.5713)(7, 0.5530)(8, 0.4624)(9, 0.4425)(10, 0.3772)(11, 0.3593)
    };
\addlegendentry{$n=4$}

\end{axis}
\end{tikzpicture}
\caption{$d=2$}
\end{subfigure}
~
\begin{subfigure}[t]{0.45\textwidth}
\centering
\begin{tikzpicture}
\begin{axis}[
    xlabel={$r$},
    ylabel={},
    grid=both,
    xmin=3, xmax=18,
    ymin=0, ymax=1,
    width=8cm,
    height=6cm,
    legend pos=south west,
    xmode=log, 
    ymode=log, 
]

\addplot[
    color=blue,
    mark=*,
    mark options={solid},
    ]
    coordinates {
        (3, 0.6001)(4, 0.5284)(5, 0.5001)(6, 0.3585)(7, 0.3140)(8, 0.2929)(9, 0.2309)(10, 0.2050)(11, 0.1910)(12, 0.1592)(13, 0.1434)(14, 0.1342)(15, 0.1179)(16, 0.1128)(17, 0.1124)(18, 0.1059)
    };
\addlegendentry{$n=1$}

\addplot[
    color=red,
    mark=square*,
    mark options={solid},
    ]
    coordinates {
        (3, 0.8278)(4, 0.7575)(5, 0.7501)(6, 0.5798)(7, 0.5263)(8, 0.5001)(9, 0.4051)(10, 0.3668)(11, 0.3455)(12, 0.2917)(13, 0.2657)(14, 0.2501)(15, 0.2175)(16, 0.1997)(17, 0.1890)(18, 0.1678)
    };
\addlegendentry{$n=2$}

\addplot[
    color=green,
    mark=triangle*,
    mark options={solid},
    ]
    coordinates {
        (3, 0.9026)(4, 0.8500)(5, 0.8334)(6, 0.6967)(7, 0.6532)(8, 0.6132)(9, 0.5456)(10, 0.4996)(11, 0.4751)(12, 0.4129)(13, 0.3823)(14, 0.3598)(15, 0.3201)(16, 0.2968)
    };
\addlegendentry{$n=3$}

\addplot[
    color=orange,
    mark=diamond*,
    mark options={solid},
    ]
    coordinates {
        (3, 0.9385)(4, 0.8919)(5, 0.8751)(6, 0.7679)(7, 0.7320)(8, 0.6876)(9, 0.6390)(10, 0.5893)(11, 0.5641)
    };
\addlegendentry{$n=4$}

\end{axis}
\end{tikzpicture}
\caption{$d=3$}
\end{subfigure}
\caption{The SDP upper bound \eqref{SDP_bound_Theta} on $\Theta_{n,d}^r$ for $d=2,3$. \label{fig:plots}}
\end{figure}

To prove \revv{Theorem \ref{proposition:Theta_SDP_bound}}, we will show that the following two statements are equivalent:
\begin{enumerate}
  \item
  $
  p(\mathbf{x}) := \sum_{\alpha \in \N^n_d}  2^{\omega(\alpha)} p_{{\alpha}}  T_{\alpha}(\mathbf{x})  \in \mathcal{T}(1\pm x_1,\dots,1\pm x_n)_{r}
  $
  \item
  $
  K_r(\mathbf{x},\mathbf{y}) := \sum_{\alpha \in \N^n_d}  p_{{\alpha}} \hat T_{\alpha}(\mathbf{x})\hat T_{\alpha}(\mathbf{y})
   \in \mathcal{T}(1\pm x_1,\dots,1\pm x_n;1\pm y_1,\dots,1\pm y_n)_{r,r}.
  $
\end{enumerate}
Note that item 1 appears in the statement of problem \eqref{SDP_bound_Theta}. Also, item 2 is a sufficient condition for \eqref{kernel_preorder_condition} to hold, again by a similar argument  as
in  \cite[Lemma 9]{Laurent2021}; \rev{ we give a
proof of this fact in Proposition \ref{prop:cubature} below.}

Also note that proving the equivalence of the two conditions is the same as proving the equivalence of  the  following two statements (after re-scaling):
\begin{enumerate}
  \item
  $
  p(\mathbf{x}) :=  \sum_{\alpha \in \N^n_d}   p_{{\alpha}}  T_{\alpha}(\mathbf{x})  \in \mathcal{T}(1\pm x_1,\dots,1\pm x_n)_{r}
  $
  \item
  $
  K_r(\mathbf{x},\mathbf{y}) :=  \sum_{\alpha \in \N^n_d}  p_{{\alpha}} T_{\alpha}(\mathbf{x}) T_{\alpha}(\mathbf{y})
   \in \mathcal{T}(1\pm x_1,\dots,1\pm x_n;1\pm y_1,\dots,1\pm y_n)_{r,r}.
  $
\end{enumerate}

\rev{
For the sake of completeness, we give a proof in Proposition \ref{prop:cubature} below that item 2 is a sufficient condition for \eqref{kernel_preorder_condition} to hold,
using a similar argument as in  \cite[Lemma 9]{Laurent2021}.
The proof relies on a celebrated theorem\footnote{The same result was obtained independently at the same time by Richter \cite{Richter}.} by Tchakaloff \cite{Tchakaloff} concerning the existence of cubature rules for numerical integration. \revv{(Modern expositions and generalizations
of Tchakaloff's theorem are given in \cite{Putinar cubature} and \cite{diDio}.)}
Recall that a positive cubature rule of degree $d$ for numerical integration with respect to a measure $\mu$ over a compact set $\K$
is defined by a set of $N$ nodes $\y^{(\ell)}\in \K$  and weights $w_\ell \ge 0$ ($\ell \in [N]$) with the property that
\[
\int_K f(\y)d\mu(\y) = \sum_{\ell=1}^N w_\ell f(\y^{(\ell)}), \quad \forall f \in \mathbb{R}[\y]_d.
\]

\begin{theorem}[Tchakaloff \cite{Tchakaloff}]
  For any compact set $\K \subset \mathbb{R}^n$, and finite, positive Borel measure supported on $\K$, there exists a positive
  cubature rule of degree $d$ with $N = {n+d \choose d}$ nodes.
\end{theorem}

One may use the Tchakaloff theorem to show the following.
\begin{proposition}
\label{prop:cubature}
Assume $K:[-1,1]^n \times [-1,1]^n \rightarrow \mathbb{R}$ and
\[
(\x,\y) \mapsto K(\x,\y) \in \mathcal{T}(1\pm x_1,\dots,1\pm x_n;1\pm y_1,\dots,1\pm y_n)_{r,r}.
\]
Then, for any $f \in \mathbb{R}[\x]_d$ that is nonnegative on $[-1,1]^n$, one has
\[
\x \mapsto \int_{[-1,1]^n} f(\y)K(\x,\y)d\mu(\y) \in \mathcal{T}(1\pm x_1,\dots,1\pm x_n)_{r},
\]
where $\mu$ is defined in \eqref{mu cheb}.
\end{proposition}
\begin{proof}
By Tchakaloff's theorem there exist
$\y^{(\ell)}\in [-1,1]^n$  and the weights $w_\ell \ge 0$ ($\ell \in [N]$) with $N = {n+d \choose d}$ such that
\[
\int_{[-1,1]^n} f(\y)K(\x,\y)d\mu(\y) = \sum_{\ell=1}^N w_\ell K(\x,\y^{(\ell)}) f(\y^{(\ell)}).
\]
Using $f(\y^{(\ell)}) \ge 0$ for all $\ell \in [N]$, as well as
\[
\x \mapsto K(\x,\y^{(\ell)}) \in \mathcal{T}(1\pm x_1,\dots,1\pm x_n)_{r} \quad \forall \ell \in [N],
\]
yields  the required result.
\end{proof}
}

For clarity of exposition, we will first prove the equivalence of the two statements in the univariate case, before proceeding to the multivariate case.

\subsection{Univariate case}
First, we work out the univariate case.
\begin{lemma}\label{lem.id}
    For any $k \ge 0$, and $x,y \in [-1,1]$, one has
    \[T_k(x)T_k(y) = \tfrac 12\left(T_k(xy + \sqrt{(1-x^2)(1-y^2)}) + T_k(xy - \sqrt{(1-x^2)(1-y^2)})\right).\]
\end{lemma}
\begin{proof}
    Let $x_1,x_2 \in [-1,1]$. Then there are $\theta_1,\theta_2 \in [0,\pi]$ such that $x_i=\cos(\theta_i)$. Then $\sin(\theta_i) \ge 0$ and thus $\sin(\theta_i) = \sqrt{1-\cos(\theta_i)^2} = \sqrt{1-x_i^2}$. Thus
    \[x_1x_2 \pm \sqrt{(1-x_1^2)(1-x_2^2)} = \cos \theta_1 \cos \theta_2 \pm \sin \theta_1 \sin\theta_2 = \cos(\theta_1 \mp \theta_2).\]
    Using $T_k(\cos(\theta)) = \cos(k \theta)$ we have
    \begin{eqnarray*}
    T_k\left(x_1x_2 + \sqrt{(1-x_1^2)(1-x_2^2)}\right)  + T_k\left(x_1x_2 - \sqrt{(1-x_1^2)(1-x_2^2)}\right)
    &=&T_k(\cos(\theta_1 - \theta_2)) + T_k(\cos(\theta_1 + \theta_2)) \\
    &=& \cos(k\theta_1 - k\theta_2) + \cos(k\theta_1 + k\theta_2) \\
    &=& \cos(k \theta_1)\cos(k \theta_2) \\
    &=&     T_k(x_1)T_k(x_2).
    \end{eqnarray*}
    \end{proof}

Now given $p \in \R[t]$  define
\[K_p(x,y) := \tfrac 12\left(p(xy + \sqrt{(1-x^2)(1-y^2)}) + p(xy - \sqrt{(1-x^2)(1-y^2)})\right).\]
From Lemma~\ref{lem.id} we obtain that if $p(t) = \sum_{k\ge 0} p_k T_k(t)$ then $K_p(x,y) = \sum_{k\ge 0} p_kT_k(x)T_k(y)$. Moreover, every kernel is of this form as
given a kernel $K(x,y) = \sum_{k \ge 0}c_k T_k(x)T_k(y)$, defining $p_K(t) = K(t,1)$ we have that \rev{$p_K(t) = \sum_{k\ge 0}c_kT_k(t)$},
since $T_k(1) = 1$ for all $k$.

We have that $|xy \pm \sqrt{(1-x^2)(1-y^2)}| \le 1$ when $|x|,|y| \le 1$.  Thus, $p$ is nonnegative on $[-1,1]$ if and only if  $K_p$ is non-negative on $[-1,1]^2$. Now, $p$ is non-negative on $[-1,1]$ if, and only if, $p \in \mathcal{T}(1\pm t)_{\deg p}$, \rev{by Theorem \ref{THM:MarkovLukacz}}. Next, we show that we also have $K_p \in \mathcal{T}(1\pm x;1 \pm y)_{\deg p, \deg p}$ when $p$ nonnegative.
To this end, we first prove the following lemma.
\begin{lemma}\label{lem:inPre} Let $d \ge 0$. Let $q \in \R[t]_d$.
\begin{enumerate}
    \item[1.] $K_{q(t)^2} \in \mathcal{T}(1\pm x;1 \pm y)_{2d,2d}$.
    \item[2.] $K_{(1 \pm t)q(t)^2} \in \mathcal{T}(1\pm x;1 \pm y)_{2d+1,2d+1}$.
    \item[3.] $K_{(1 - t^2)q(t)^2} \in \mathcal{T}(1\pm x;1 \pm y)_{2d+2,2d+2}$.
 \end{enumerate}
\end{lemma}
\begin{proof} First notice that by expanding and collecting terms, we can write
\[q(xy + \sqrt{(1-x^2)(1-y^2)}) = q_0(x,y) + \sqrt{(1-x^2)(1-y^2)} q_1(x,y)\]
and
\[q(xy - \sqrt{(1-x^2)(1-y^2)}) = q_0(x,y) - \sqrt{(1-x^2)(1-y^2)} q_1(x,y)\]
with $q_0 \in \R[x,y]_{d,d}$ and $q_1 \in \R[x,y]_{d-1,d-1}$.

\medskip
We then have
\begin{equation}\label{eq:SOStrans}
q(xy \pm \sqrt{(1-x^2)(1-y^2)})^2 = q_0^2 + (1-x^2)(1-y^2)q_1^2 \pm 2\sqrt{(1-x^2)(1-y^2)}q_0q_1.
\end{equation}

From~\eqref{eq:SOStrans}
 $K_{q^2} = q_0^2 + (1-x^2)(1-y^2)q_1^2$ follows, proving 1. Now, to prove 2, we write
\begin{align*}
2K_{(1 - t)q(t)^2}
&:= (1- xy - \sqrt{(1-x^2)(1-y^2)})q(xy + \sqrt{(1-x^2)(1-y^2)})^2 \\
&\qquad +(1- xy + \sqrt{(1-x^2)(1-y^2)})q(xy -\sqrt{(1-x^2)(1-y^2)})^2\\
&= (1- xy - \sqrt{(1-x^2)(1-y^2)})(q_0^2 + (1-x^2)(1-y^2)q_1^2 + 2\sqrt{(1-x^2)(1-y^2)}q_0q_1) \\
&\qquad+ (1- xy + \sqrt{(1-x^2)(1-y^2)})(q_0^2 + (1-x^2)(1-y^2)q_1^2 -2\sqrt{(1-x^2)(1-y^2)}q_0q_1),
\end{align*}
\rev{where the second equality follows from~\eqref{eq:SOStrans}. }

\rev{Using $(a-b)(c+d) + (a+b)(c-d) = 2ac - 2bd$ we obtain,
\begin{align*}
2K_{(1 - t)q(t)^2}
&= 2(1- xy)(q_0^2 + (1-x^2)(1-y^2)q_1^2) - 4(1-x^2)(1-y^2)q_0q_1.\\
\end{align*}
Let $A = (1-x)(1+y)$ and $B = (1+x)(1-y)$. We have then $A+B = 2(1- xy)$ and $AB = (1-x^2)(1-y^2)$. Therefore,
\begin{align*}
2K_{(1 - t)q(t)^2}
&= (A+B)(q_0^2 + AB q_1^2) - 4ABq_0q_1 = A(q_0-Bq_1)^2 + B(q_0-Aq_1)^2 \\
&= (1-x)(1+y)(q_0 - (1+x)(1-y)q_1)^2 + (1+x)(1-y)(q_0 - (1-x)(1+y)q_1)^2.
\end{align*}
}
Similarly,
\[2K_{(1 + t)q(t)^2} =(1+x)(1+y)(q_0 + (1-x)(1-y) q_1)^2 + (1-x)(1-y)(q_0 + (1+x)(1+y)q_1)^2.
\]
To prove 3, we write,
\begin{align*}
2K_{(1 - t^2)q(t)^2}
&:= \left(1- (xy + \sqrt{(1-x^2)(1-y^2)})^2\right)q(xy + \sqrt{(1-x^2)(1-y^2)})^2 \\
&\qquad +\left(1- (xy - \sqrt{(1-x^2)(1-y^2)})^2\right)q(xy -\sqrt{(1-x^2)(1-y^2)})^2\\
&= \left(1- (xy)^2  - (1-x^2)(1-y^2)  - 2xy\sqrt{(1-x^2)(1-y^2)}\right) \\
&\qquad\qquad \cdot \left(q_0^2 + (1-x^2)(1-y^2)q_1^2 + 2\sqrt{(1-x^2)(1-y^2)}q_0q_1\right) \\
&\qquad+ \left(1- (xy)^2  -  (1-x^2)(1-y^2)  + 2xy\sqrt{(1-x^2)(1-y^2)}\right) \\
&\qquad\qquad\qquad \cdot \left(q_0^2 + (1-x^2)(1-y^2)q_1^2 -2\sqrt{(1-x^2)(1-y^2)}q_0q_1\right),
\end{align*}
\rev{where we have used~\eqref{eq:SOStrans} to obtain the last equality.

 Using $(a-b)(c+d) + (a+b)(c-d) = 2ac - 2bd$ again, we obtain,
\begin{align*}
2K_{(1 - t^2)q(t)^2}
&= 2\left(1- (xy)^2  - (1-x^2)(1-y^2) \right) (q_0^2 + (1-x^2)(1-y^2)q_1^2) - 8xy(1-x^2)(1-y^2)q_0q_1.\\
\end{align*}
Now, let $A = (1-x^2)$ and $B = (1-y^2)$ we have then $(1-x^2)(1-y^2) = AB$ and $1- (xy)^2 = AB + x^2B + y^2A$. Therefore,
\begin{align*}
K_{(1 - t^2)q(t)^2} &= (x^2B + y^2A)(q_0^2 + ABq_1^2) - 4xyABq_0q_1 \\
& = B(xq_0 - y Aq_1)^2 +  A(yq_0 - xBq_1)^2 \\
& = (1-y^2)(xq_0 - y (1-x^2)q_1)^2 +  (1-x^2)(yq_0 - x(1-y^2)q_1)^2.
\end{align*}
}
\end{proof}
Now we are ready to prove our main result in this section.
\begin{theorem}\label{Thm:univariate} Let $p \in \mathcal{T}(1\pm t)_r$. Then $K_p \in \mathcal{T}(1\pm x;1\pm y)_{r,r}$.
\end{theorem}
\begin{proof}
We have $p =\sigma_0 +  (1-t)\sigma_1 + (1+t)\sigma_2 + (1-t^2)\sigma_3$ where $\sigma_i \in \Sigma[t]$ and $\deg(\sigma_0), \deg(\sigma_1)+1, \deg(\sigma_2)+1, \deg(\sigma_3)+2 \le r$. By the linearity of the operator $p \mapsto K_p$ and Lemma~\ref{lem:inPre}, we obtain $K_p \in \mathcal{T}(1\pm x,1\pm y)_{r,r}$.
\end{proof}

\begin{corollary}
  Let $K(x,y) = \sum_{k=0}^r c_kT_k(x)T_k(y)$ be a nonnegative kernel. Then $K \in \mathcal{T}(1\pm x,1\pm y)_{r,r}$.
\end{corollary}
\begin{proof}
Let $p(t)=K(t,1)$. We have then $K = K_p$ and $p$ is nonnegative on $[-1,1]$.  \rev{By Theorem~\ref{THM:MarkovLukacz}, we have} $p \in \mathcal{T}(1\pm t)_{r}$, and by Theorem~\ref{Thm:univariate} the statement follows.
\end{proof}

\subsection{Multivariate case}
The natural question is whether the results from the previous section translate to the general case. Given an $n$-variate polynomial
$p \in \R[\varu]$, it is natural to define
\[
K_p(\x,\y) = \tfrac 1{2^n}\sum p(x_1y_1\pm \sqrt{(1-x_1^2)(1-y_1^2)},\dots,x_ny_n\pm \sqrt{(1-x_n^2)(1-y_n^2)}),
\]
where the summation is over all the possible $2^n$ sign combinations.

As in the univariate case, from Lemma~\ref{lem.id} we obtain that if $p(\varu) = \sum_{\alpha} p_\alpha T_\alpha(\varu)$ then
$K_p(\x,\y) = \sum_{\alpha} p_\alpha T_\alpha(\x)T_\alpha(\y)$. Moreover, every kernel is of this form,
 as, given a kernel $K(\x,\y) = \sum_{\alpha}c_\alpha T_\alpha(\x)T_\alpha(\y)$, we have $K = K_p$ for
 $p(\varu) = K(\varu,\mathbf{1})$, where $\mathbf{1}$ denotes the all-ones vector.

We have then,  $p$ nonnegative on $[-1,1]^n$ if and only if  $K_p$ non-negative on $[-1,1]^{2n}$.
 Next we prove $p \in \mathcal{T}(1\pm u_1,\dots,1\pm u_n)_{r}$ implies
  $K_p \in \mathcal{T}(1\pm x_1,\dots,1\pm x_n;1 \pm y_1,\dots,1 \pm y_n)_{r,r}.$ The proof is basically the same as
   in the univariate case but requires induction. \rev{To do this, for each $i=1,\dots,n$ we
define the operator $\kappa^i:\R[\varu,\x,\y] \to \R[\varu,\x,\y]$} by
\rev{\begin{align*}
p(u_1,\dots,u_{i},\dots,u_n,\x,\y)
&\mapsto \kappa^i(p)(\varu,\x,\y)\\
&\qquad := \tfrac 1{2}\left(p(u_1,\dots,u_{i-1},x_iy_i +\sqrt{(1-x_i^2)(1-y_i^2)},u_{i+1},\dots,u_n,\x,\y)\right. \\
&\qquad\qquad + \left.p(u_1,\dots,u_{i-1},x_iy_i -\sqrt{(1-x_i^2)(1-y_i^2)},u_{i+1},\dots,u_n,\x,\y)\right).
\end{align*}
}
Now, given $p \in \R[\varu]$ define $P^i \in \R[\varu,\x,\y]$ for $i=0,\dots,n$ by
$P^0 = p$, and $P^{i} = \kappa^{i}(P^{i-1})$ for $i=1,\dots,n$.\rev{ Notice that in the $i$-th step the variable $u_i$ is eliminated and the variables $x_i$ and $y_i$ introduced. Moreover,} we obtain that $K_p(\x,\y) = P^n$.
Now we rewrite the statement of Lemma~\ref{lem:inPre} in a bit more general form.
\begin{lemma}\label{lem:inPreGen}
\rev{
Let $q \in \R[\varu,\x,\y]$.
Then, there are $q_0 \in \R[\varu,\x,\y]$ and $q_1 \in \R[\varu,\x,\y]$ such that
\begin{equation}\label{eq:newSOStrans}
q(u_1,\dots, u_{i-1}, x_iy_i \pm \sqrt{(1-x_i^2)(1-y_i^2)},u_{i+1},\dots,u_n,\x,\y) = q_0(\varu,\x,\y) \pm \sqrt{(1-x_i^2)(1-y_i^2)} q_1(\varu,\x,\y),
\end{equation}
with $\deg_{\varu \x}(q_0) \le \deg_{\varu \x}(q)$, $\deg_{\varu \y}(q_0) \le \deg_{\varu \y}(q)$,
 $\deg_{\varu \x}(q_1) \le \deg_{\varu \x}(q) - 1$, and $\deg_{\varu \y}(q_1) \le \deg_{\varu \y}(q) - 1$.

Moreover, we have
\begin{align*}
 \kappa^i(q^2) & =  q_0^2 + (1-x_i^2)(1-y_i^2)q_1^2,\\
 2\kappa^i((1 - u_i)q^2) &=  (1-x_i)(1+y_i)(q_0 - (1+x_i)(1-y_i) q_1)^2  + (1+x_i)(1-y_i)(q_0 - (1-x_i)(1+y_i)q_1)^2,\\
  2\kappa^i((1 + u_i)q^2) &= (1+x_i)(1+y_i)(q_0 + (1-x_i)(1-y_i) q_1)^2 +  (1-x_i)(1-y_i)(q_0 + (1+x_i)(1+y_i)q_1)^2,\\
    \kappa^i((1 - u_i^2)q^2) &= (1-y_i^2)(x_iq_0 - y_i(1-x_i^2)q_1)^2 +  (1-x_i^2)(y_iq_0 - x_i(1-y_i^2)q_1)^2.
 \end{align*}
}
\end{lemma}
\begin{proof}\rev{ First notice that by expanding and collecting terms, we obtain $q_0$ and $q_1$ satisfying all the degree requirements.}
 The rest of the statement is proven  following the proof of Lemma~\ref{lem:inPre} line by line  after equation~\eqref{eq:SOStrans}, but only using equation~\eqref{eq:newSOStrans} instead of~\eqref{eq:SOStrans}.
\end{proof}

Now we can prove the main result of this section.

\begin{proposition}
\label{prop:equiv formulations} \rev{Let $p \in \R[\varu]$ be given. Then,}
$p \in \mathcal{T}(1\pm u_1,\dots,1\pm u_n)_r$ \rev{if and only if} $K_p \in \mathcal{T}(1\pm x_1,\dots,1 \pm x_n;1\pm y_1,1\pm y_n)_{r,r}$.
\end{proposition}
\begin{proof} \rev{Assume $K_p \in \mathcal{T}(1\pm x_1,\dots,1 \pm x_n;1\pm y_1,1\pm y_n)_{r,r}$, then as $p(\varu) = K(\varu,\mathbf{1})$ we obtain $p \in \mathcal{T}(1\pm u_1,\dots,1\pm u_n)_r$.
Now, we prove the other direction. Assume $p \in \mathcal{T}(1\pm u_1,\dots,1\pm u_n)_r$ \revv{is given}.}
Define $P^i \in \R[\varu,\x,\y]$ for $i=0,\dots,n$ by $P^0 = p$, and $P^{i} = \kappa^{i}(P^{i-1})$ for $i=1,\dots,n$. We obtain then that $K_p(\x,\y) = P^n$.
\rev{
Write
\[P^0(\varu,\x,\y) = p(\varu) = \sum_{J^-,J^+ \subseteq [n]}
\sigma_{J^-,J^+}(\varu)\prod_{j \in J^-}(1 - u_j)\prod_{j \in J^+}(1 + u_j),\]
where, for all $J^-,J^+ \subseteq [n]$,  $ \sigma_{J^-,J^+} \in \Sigma[\varu]$, and
$|J^-|+|J^+|+ \mbox{deg}\sigma_{J^-,J^+} \le r.$

Notice that for all $i\in [n]$ and $q_1,q_2 \in \R[\varu,\x,\y]$ such that $u_i$ does not appear in $q_2$ (that is $\deg_{u_i}(q_2) =  0$), we have
from the definition of $\kappa^i$, that  $\kappa^i(q_1q_2) = q_2\kappa^i(q_1)$. Thus by linearity of $\kappa^i$ and  Lemma~\ref{lem:inPreGen}  we obtain by induction,
\[
P^i = \sum_{\substack{J^-,J^+ \subseteq [n]\setminus [i]\\I_x^-,I_x^+ \subseteq [i]\\I_y^-,I_y^+\subseteq[i]}}\sigma_{J^-,J^+,I_x^-,I_x^+,I_y^-,I_y^+} \prod_{j \in J^-}(1 - u_j)\prod_{j \in J^+}(1 + u_j),\prod_{j \in I^-_x}(1 - x_j)\prod_{j \in I^+_x}(1 + x_j)\prod_{j \in I^-_y}(1 - y_j)\prod_{j \in I^+_y}(1 + y_j),\]
where, for all $J^-,J^+ \subseteq [n]\setminus [i]$, all $I_x^-,I_x^+ \subseteq [i]$ and all $I_y^-,I_y^+\subseteq[i]$ we have  $ \sigma_{J^-,J^+,I_x^-,I_x^+,I_y^-,I_y^+} \in \Sigma[\varu,\x,\y]$, and
$|J^-|+|J^+|+ |I_x^-| + |I_x^+| + \deg_{\varu \x}\sigma_{J^-,J^+,I_x^-,I_x^+,I_y^-,I_y^+} \le r$ and $|J^-|+|J^+|+ |I_y^-| + |I_y^+| + \deg_{\varu \y}\sigma_{J^-,J^+,I_x^-,I_x^+,I_y^-,I_y^+} \le r$.

In particular,  $K_p = P^n \in \mathcal{T}(1\pm x_1,\dots,1 \pm x_n;1\pm y_1,1\pm y_n)_{r,r}$.
}
\end{proof}

\rev{
We now conclude the section with \revv{a  proof of Theorem \ref{proposition:Theta_SDP_bound}}.

\subsection*{Proof of  Theorem \ref{proposition:Theta_SDP_bound}}
\begin{proof}
Assume that we have a feasible solution of problem \eqref{SDP_bound_Theta}, i.e.\ we have a polynomial in the truncated pre-ordering, say
\[
\mathbf{x} \mapsto 1+ \sum_{\alpha \in \N^n_d\setminus \{\mathbf{0}\}}  2^{\omega(\alpha)} p_{{\alpha}}  T_{\alpha}(\mathbf{x})  \in \mathcal{T}(1\pm x_1,\dots,1\pm x_n)_{r}.
\]
By Proposition \ref{prop:equiv formulations}, we then have that
\[
(\mathbf{x},\y) \mapsto 1+ \sum_{\alpha \in \N^n_d\setminus \{\mathbf{0}\}} p_{{\alpha}} \hat T_{\alpha}(\mathbf{x})\hat T_{\alpha}(\mathbf{y})
   \in \mathcal{T}(1\pm x_1,\dots,1\pm x_n;1\pm y_1,\dots,1\pm y_n)_{r,r}.
\]
By Proposition \ref{prop:cubature}, we have that this kernel satisfies property \eqref{kernel_preorder_condition}.
Thus this kernel is feasible for the problem \eqref{SDP:kappa_r}. The required result follows.
\end{proof}
}

\section{Concluding remarks}
\label{sec:conclusion}
Our main results may best be framed in terms of error bounds for the SDP hierarchy \eqref{Schmuedgen_SDP2}.
In particular, we have shown the following sequence of inequalities:
\begin{eqnarray*}
   f_\min - \hat f_{(r)} & \le &\Theta^{r}_{n,d}\|f\|_{1, T} \quad \mbox{(by Theorem \ref{thm:error_bound1})} \\
   &\le & \frac {\pi^2 d^2n^2}{r^2}\|f\|_{1, T} \quad \mbox{(by Proposition \ref{prop:Theta_bound_nr} if $n|r$ and $r  \ge \pi dn$)} \\
    &\le & \frac {\pi^2 d^2n^2}{r^2}\|f\|_{1} \quad \mbox{(by \eqref{one_norm_conversion})} \\
   &\le &  \frac {\pi^2 d^2n^2}{r^2}{n+d \choose d}^{ \frac{1}{2}}C_{d}\|f\|_{\ell_\infty} \quad \mbox{(by Bohnenblust-Hille ineq.\ \eqref{Bohnenblust-Hille})} \\
   &\le & \frac {\pi^2 d^2n^2}{r^2}{n+d \choose d}^{ \frac{1}{2}}C_{d}(f_\max - f_\min) \quad \mbox{(w.l.o.g.\ if $f(0) = 0$)}. \\
\end{eqnarray*}
We may readily compare the outermost inequality to the result by Laurent and Slot \cite{Laurent2021}, as reproduced in Theorem \ref{thm:Laurent_Slot2},
to  see that we obtain a very similar result. The added value of our analysis is therefore in the first inequality, since,
 for special classes of polynomials, $\|f\|_{1,T}$ may relate more favorably to $\|f\|_{\ell_\infty}$ than in the general case.
 \rev{
 We should mention that this is not the first time where changing or comparing norms gives new insights in polynomial optimization. For example, in \cite{Baldi_et_al}
    the authors improve the known dependence of the general Putinar's theorem \cite{Putinar} on the number of variables, by switching from the supremum norm to the $1$-norm in the Bernstein basis.
 }

The object $\Theta^{r}_{n,d}$ is of independent interest, as it corresponds to families of positive approximation operators, and may be upper bounded
by an SDP as in \revv{Theorem \ref{proposition:Theta_SDP_bound}}. A key question is if one could find an analytical expression for (an upper bound on)
$\Theta^{r}_{n,d}$ in terms of $r,n,d$ in a way that does not involve the product of univariate (Jackson) kernels.

Our research links two, seemingly different, topics:
\begin{enumerate}
  \item
  Deriving error bounds for SDP hierarchies for polynomial optimization, in the spirit of \cite{Laurent2021} and \cite{Bach_Rudi}.
  \item
  Constructing optimal multivariate approximation kernels using SDP, in the spirit of \cite{Kirchsner2023}.
\end{enumerate}
In particular, we  showed how kernels that are optimized for $1$-norm polynomial approximation on the hypercube yield error bounds
for SDP hierarchies of the Schm\"udgen type.

A natural question is if our analysis may be extended to SDP hierarchies that use the Putinar Positivstellensatz \cite{Putinar} instead of the one by Schm\"udgen.
This is of theoretical as well as computational interest, since the famous Lasserre SDP hierachy \cite{lass-siopt-01} for polynomial optimization is based on the Putinar Positivstensatz.
In particular, Baldi and Slot \cite{Baldi_Slot} have recently extended the results by Laurent and Slot \cite{Laurent2021} to obtain the best-known effective version of the Putinar Positivstennsatz. It would be interesting to investigate if our approach allows a similar extension. \rev{The main obstacle is that our key technical result, namely Corollary \ref{cor:Juan}, only applies to pre-orderings, and not to quadratic modules.}

\rev{
\subsection*{Acknowledments}
The authors would like to thank Monique Laurent for valuable feedback on the first version of this manuscript, and two anonymous referees
for their helpful comments that led to an improved presentation of the results.
}

\section*{Appendix: numerical SDP bounds}
In this appendix we present the SDP upper bound \eqref{SDP_bound_Theta} on $\Theta_{n,d}^r$ for different values of $(n,r,d)$, as indicated in the tables.
We may mention that, when $r$ is a multiple of $n$ in the tables, the SDP upper
 bound on $\Theta_{n,d}^r$ is strictly lower than the value obtained from products of univariate Jackson kernels.
 In particular, for the appropriate values the tables where $r$ is a multiple of $n$, the inequality in Proposition \ref{prop:Theta_bound_nr} is strict.
 The values in the tables were computed using the SDP solver MOSEK \cite{mosek9} under Matlab.
\begin{table}[h!]
\centering
\begin{adjustbox}{width=\textwidth}
\begin{tabular}{ ccccccccccccccccccccccccccccccc}
 & d=1 & 2 & 3 & 4 & 5 & 6 & 7 & 8 & 9 & 10 & 11 & 12 & 13 & 14 & 15  \\
 \hline
r=1 & 0.5000 \\
2 & 0.2929 & 0.5556  \\
3 & 0.1910 & 0.5001 & 0.6001  \\
4 & 0.1340 & 0.3320 & 0.5284 & 0.6214  \\
5 & 0.0991 & 0.2929 & 0.5001 & 0.5556 & 0.6405  \\
6 & 0.0762 & 0.2159 & 0.3585 & 0.5177 & 0.5783 & 0.6527  \\
7 & 0.0604 & 0.1910 & 0.3140 & 0.5001 & 0.5374 & 0.6001 & 0.6641  \\
8 & 0.0490 & 0.1502 & 0.2929 & 0.3723 & 0.5123 & 0.5556 & 0.6108 & 0.6723  \\
9 & 0.0406 & 0.1340 & 0.2309 & 0.3320 & 0.5001 & 0.5284 & 0.5707 & 0.6214 & 0.6802  \\
10 & 0.0341 & 0.1101 & 0.2050 & 0.3066 & 0.3834 & 0.5090 & 0.5418 & 0.5848 & 0.6310 & 0.6862  \\
11 & 0.0291 & 0.0991 & 0.1910 & 0.2929 & 0.3461 & 0.5001 & 0.5217 & 0.5556 & 0.6001 & 0.6405 & 0.6921  \\
12 & 0.0251 & 0.0839 & 0.1592 & 0.2389 & 0.3202 & 0.3909 & 0.5070 & 0.5336 & 0.5668 & 0.6070 & 0.6467 & 0.6968  \\
13 & 0.0219 & 0.0762 & 0.1434 & 0.2160 & 0.3027 & 0.3587 & 0.5001 & 0.5177 & 0.5447 & 0.5793 & 0.6152 & 0.6542 & 0.7015  \\
14 & 0.0193 & 0.0660 & 0.1342 & 0.2012 & 0.2930 & 0.3341 & 0.4169 & 0.5056 & 0.5311 & 0.5606 & 0.6041 & 0.6271 & 0.6960 & 0.7603  \\
15 & 0.0171 & 0.0604 & 0.1179 & 0.1925 & 0.2570 & 0.3188 & 0.3933 & 0.5001 & 0.5159 & 0.5409 & 0.5772 & 0.6120 & 0.6311 & 0.6904 & 0.9302  \\
16 & 0.0153 & 0.0552 & 0.1128 & 0.1830 & 0.2544 & 0.3114 & 0.3764 & 0.4940 & 0.5147 & 0.5434 & 0.5781 & 0.6136 & 0.6471 & 0.8190 & 0.9657  \\
17 & 0.0137 & 0.0539 & 0.1124 & 0.1804 & 0.2549 & 0.3061 & 0.3706 & 0.4850 & 0.5072 & 0.5323 & 0.5644 & 0.5944 & 0.6284 & 0.7332 & 0.8572  \\
18 & 0.0137 & 0.0508 & 0.1059 & 0.1712 & 0.2273 & 0.2978 & 0.3590 & 0.4652 & 0.5030 & 0.5282 & 0.5776 & 0.5872 & 0.6259 & 0.6649 & 0.8315  \\
\hline
 \end{tabular}
 \end{adjustbox}
\caption{The SDP upper bound \eqref{SDP_bound_Theta} on $\Theta_{n,d}^r$ for $n = 1$}
\end{table}
\begin{table}[h!]
\centering
\begin{adjustbox}{width=\textwidth}
\begin{tabular}{ ccccccccccccccccccccccccccccccc}
 & d=1 & 2 & 3 & 4 & 5 & 6 & 7 & 8 & 9 & 10 & 11 & 12 & 13 & 14 & 15 & \\
 \hline
 r=1 & 0.7500  \\
2 & 0.5001 & 0.7648  \\
3 & 0.3455 & 0.7501 & 0.8278  \\
4 & 0.2501 & 0.5295 & 0.7575 & 0.8338  \\
5 & 0.1883 & 0.5001 & 0.7501 & 0.8025 & 0.8620  \\
6 & 0.1465 & 0.3720 & 0.5798 & 0.7566 & 0.8212 & 0.8640  \\
7 & 0.1170 & 0.3455 & 0.5263 & 0.7501 & 0.7859 & 0.8401 & 0.8805  \\
8 & 0.0955 & 0.2706 & 0.5001 & 0.5920 & 0.7557 & 0.7968 & 0.8459 & 0.8800  \\
9 & 0.0794 & 0.2501 & 0.4051 & 0.5537 & 0.7501 & 0.7776 & 0.8142 & 0.8566 & 0.8918  \\
10 & 0.0670 & 0.2038 & 0.3668 & 0.5167 & 0.6166 & 0.7549 & 0.7882 & 0.8219 & 0.8633 & 0.8912  \\
11 & 0.0573 & 0.1883 & 0.3455 & 0.5001 & 0.5722 & 0.7501 & 0.7712 & 0.8025 & 0.8359 & 0.8708 & 0.8999  \\
12 & 0.0496 & 0.1581 & 0.2917 & 0.4140 & 0.5378 & 0.6191 & 0.7542 & 0.7801 & 0.8123 & 0.8413 & 0.8747 & 0.8989  \\
13 & 0.0433 & 0.1465 & 0.2657 & 0.3852 & 0.5132 & 0.5884 & 0.7501 & 0.7674 & 0.7925 & 0.8222 & 0.8503 & 0.8795 & 0.9063  \\
14 & 0.0381 & 0.1259 & 0.2501 & 0.3592 & 0.5001 & 0.5501 & 0.6340 & 0.7547 & 0.7761 & 0.8018 & 0.8304 & 0.8581 & 0.8833 & 0.9058  \\
15 & 0.0338 & 0.1170 & 0.2175 & 0.3455 & 0.4288 & 0.5294 & 0.5986 & 0.7501 & 0.7645 & 0.7865 & 0.8104 & 0.8403 & 0.8606 & 0.8902 & 0.9283  \\
16 & 0.0302 & 0.1025 & 0.1997 & 0.3039 & 0.4001 & 0.5114 & 0.5678 & 0.6644 & 0.7544 & 0.7720 & 0.7983 & 0.8213 & 0.8535 & 0.8728 & 0.9043  \\
17 & 0.0271 & 0.0955 & 0.1890 & 0.2810 & 0.3759 & 0.5007 & 0.5440 & 0.6154 & 0.7502 & 0.7646 & 0.7917 & 0.8089 & 0.8388 & 0.8604 & 0.8881  \\
18 & 0.0245 & 0.0877 & 0.1678 & 0.2659 & 0.3600 & 0.4742 & 0.5336 & 0.6081 & 0.6705 & 0.7560 & 0.7771 & 0.8029 & 0.8307 & 0.8489 & 0.8624  \\
\hline
 \end{tabular}
\end{adjustbox}
\caption{The SDP upper bound \eqref{SDP_bound_Theta} on $\Theta_{n,d}^r$ for $n = 2$}
\end{table}
\begin{table}[h!]
\centering
\begin{adjustbox}{width=\textwidth}
\begin{tabular}{ ccccccccccccccccccccccccccccccc}
 & d=1 & 2 & 3 & 4 & 5 & 6 & 7 & 8 & 9 & 10 & 11 & 12 & 13 & 14 & 15 & \\
 \hline
r=1& 0.8333 \\
2 & 0.5918 & 0.8401  \\
3 & 0.4727 & 0.8334 & 0.9026  \\
4 & 0.3546 & 0.6331 & 0.8500 & 0.9048  \\
5 & 0.2787 & 0.6092 & 0.8334 & 0.8889 & 0.9319  \\
6 & 0.2197 & 0.4952 & 0.6967 & 0.8408 & 0.9016 & 0.9322  \\
7 & 0.1785 & 0.4727 & 0.6532 & 0.8334 & 0.8795 & 0.9201 & 0.9464  \\
8 & 0.1467 & 0.3804 & 0.6132 & 0.7150 & 0.8410 & 0.8847 & 0.9246 & 0.9455  \\
9 & 0.1230 & 0.3582 & 0.5456 & 0.6925 & 0.8334 & 0.8732 & 0.9039 & 0.9343 & 0.9547  \\
10 & 0.1042 & 0.2961 & 0.4996 & 0.6374 & 0.7380 & 0.8402 & 0.8797 & 0.9073 & 0.9378 & 0.9537  \\
11 & 0.0895 & 0.2788 & 0.4751 & 0.6189 & 0.7097 & 0.8334 & 0.8661 & 0.8959 & 0.9197 & 0.9439 & 0.9602  \\
12 & 0.0776 & 0.2348 & 0.4129 & 0.5523 & 0.6682 & 0.7461 & 0.8410 & 0.8716 & 0.9027 & 0.9221 & 0.9462 & 0.9591  \\
13 & 0.0679 & 0.2207 & 0.3823 & 0.5289 & 0.6421 & 0.7260 & 0.8335 & 0.8621 & 0.8887 & 0.9118 & 0.9305 & 0.9503 & 0.9642  \\
14 & 0.0599 & 0.1897 & 0.3598 & 0.4921 & 0.6181 & 0.6826 & 0.7598 & 0.8402 & 0.8686 & 0.8927 & 0.9161 & 0.9316 & 0.9536 & 0.9634  \\
15 & 0.0533 & 0.1786 & 0.3201 & 0.4759 & 0.5752 & 0.6643 & 0.7374 & 0.8340 & 0.8585 & 0.8838 & 0.9029 & 0.9228 & 0.9385 & 0.9555 & 0.9677  \\
16 & 0.0477 & 0.1560 & 0.2968 & 0.4278 & 0.5470 & 0.6348 & 0.7039 & 0.7671 & 0.8479 & 0.8647 & 0.8911 & 0.9069 & 0.9282 & 0.9444 & 0.9591  \\
\hline
 \end{tabular}
\end{adjustbox}
\caption{The SDP upper bound \eqref{SDP_bound_Theta} on $\Theta_{n,d}^r$ for $n = 3$}
\end{table}

\begin{table}[h!]
\centering
\begin{adjustbox}{width=0.75\textwidth}
\begin{tabular}{ ccccccccccccccccccccccccccccccc}
 & d=1 & 2 & 3 & 4 & 5 & 6 & 7 & 8 & 9 & 10 & 11  \\
 \hline
r=1 & 0.8750 \\
 2 & 0.6465 & 0.8788  \\
3 & 0.5510 & 0.8751 & 0.9385  \\
4 & 0.4272 & 0.6943 & 0.8919 & 0.9376  \\
5 & 0.3575 & 0.6770 & 0.8751 & 0.9286 & 0.9613  \\
6 & 0.2866 & 0.5713 & 0.7679 & 0.8841 & 0.9380 & 0.9604  \\
7 & 0.2398 & 0.5530 & 0.7320 & 0.8788 & 0.9226 & 0.9543 & 0.9723  \\
8 & 0.1985 & 0.4624 & 0.6876 & 0.7828 & 0.8883 & 0.9255 & 0.9575 & 0.9713  \\
9 & 0.1689 & 0.4425 & 0.6390 & 0.7680 & 0.8779 & 0.9183 & 0.9443 & 0.9662 & 0.9784  \\
10 & 0.1436 & 0.3773 & 0.5893 & 0.7122 & 0.8087 & 0.8854 & 0.9235 & 0.9453 & 0.9679 & 0.9776  \\
11 & 0.1244 & 0.3593 & 0.5641 & 0.6976 & 0.7873 & 0.8794 & 0.9130 & 0.9388 & 0.9562 & 0.9731 & 0.9821  \\
\hline
 \end{tabular}
\end{adjustbox}
\caption{The SDP upper bound \eqref{SDP_bound_Theta} on $\Theta_{n,d}^r$ for $n = 4$}
\end{table}


	\end{document}